\documentclass[10pt]{amsart}

\usepackage{amssymb}

\usepackage{graphicx}

\usepackage[cmtip,all]{xy}

\usepackage{}
\usepackage{bm}       
\usepackage{bbm}       
\usepackage{mathtools} 
\usepackage{colonequals}
\usepackage{numprint} 
\usepackage[shortlabels]{enumitem}
\usepackage{chngcntr}
\usepackage{pgfplots}
\pgfplotsset{compat=1.15}
\usepackage{tikz}
\usepackage{tikz-cd}
\usepackage{xcolor}
\usepackage{color}
\usepackage[pdfencoding=auto, psdextra, breaklinks=true]{hyperref}
\definecolor{mylinkcolor}{rgb}{0.5,0.0,0.0}
\definecolor{myurlcolor}{rgb}{0.0,0.0,0.7}
\hypersetup{
 colorlinks,
 urlcolor=myurlcolor,
 citecolor=myurlcolor,
 linkcolor=mylinkcolor,
 breaklinks=true,
 pdfauthor={Shiva Chidambaram}
}
\usepackage[capitalize]{cleveref}

\newtheorem{theorem}{Theorem}[section]

\theoremstyle{definition}
\newtheorem{definition}[theorem]{Definition}
\newtheorem{example}[theorem]{Example}

\newtheorem{algorithm}[theorem]{Algorithm}
\newtheorem{prop}[theorem]{Proposition}

\theoremstyle{remark}
\newtheorem{remark}[theorem]{Remark}
\newtheorem{question}[theorem]{Question}

\numberwithin{equation}{section}

\def\Gal{\mathrm{Gal}}
\def\Jac{\mathrm{Jac}}

\def\Frob{\mathrm{Frob}}

\DeclareMathOperator{\GL}{GL}
\DeclareMathOperator{\GSp}{GSp}
\DeclareMathOperator{\Sp}{Sp}

\newcommand{\MA}{\mathcal{A}}

\newcommand{\F}{\mathbb{F}}

\newcommand{\Q}{\mathbb{Q}}

\newcommand{\Z}{\mathbb{Z}}
\newcommand{\calX}{\mathcal{X}}
\newcommand{\calK}{\mathcal{K}}

\newcommand{\calJ}{\mathcal{J}}
\newcommand{\Fbar}{\overline{\F}}
\newcommand{\Qbar}{\overline{\Q}}

\newcommand\im{\mathrm{im}}
\newcommand\rhobar{\overline{\rho}}

\newcommand{\tsim}{\text{sim}}

\DeclareFontFamily{U}{wncy}{}
\DeclareFontShape{U}{wncy}{m}{n}{<->wncyr10}{}
\DeclareSymbolFont{mcy}{U}{wncy}{m}{n}
\DeclareMathSymbol{\Sh}{\mathord}{mcy}{"58}

\begin{document}

\title[Mod-$3$ Galois images of principally polarized abelian surfaces]{Computing the mod-$3$ Galois image of a principally polarized abelian surface over the rationals}


\author[S. Chidambaram]{Shiva Chidambaram}
\address{Department of Mathematics, University of Wisconsin-Madison, Madison WI 53706-1325, USA}
\email{chidambaram3@wisc.edu}
\thanks{The author was supported by Simons Foundation grant 550033.}



\date{}

\begin{abstract}
    A lot of work has gone into computing images of Galois representations coming from elliptic curves. This article presents an algorithm to determine the image of the mod-$3$ Galois representation associated to a principally polarized abelian surface over $\Q$. Conjugacy class distribution of subgroups of $\GSp(4,\F_3)$ is a key ingredient. While this ingredient is feasible to compute for $\GSp(4,\F_\ell)$ for any small prime $\ell$, distinguishing Gassmann-equivalent subgroups is a delicate problem. We accomplish it for $\ell = 3$ using several techniques. The algorithm does not require the knowledge of endomorphisms.
\end{abstract}

\maketitle


\section{Introduction}
\label{sec:intro}
Given an elliptic curve $E/\Q$, it is a fundamental problem to understand the Galois action on its $\ell$-torsion points for a prime $\ell$. This leads to the concrete question of determining the image of the associated mod-$\ell$ Galois representation $\rhobar_{E,\ell} : G_{\Q} \rightarrow \GL(2,\F_\ell)$ up to conjugacy in $\GL(2,\F_\ell)$. There are efficient algorithms \cite{Sut16} accomplishing this task. Moreover, when $E$ does not have complex multiplication, the work \cite{RSZB22} determines the $\ell$-adic Galois image, and Zywina goes further and computes the full adelic Galois image \cite{Zyw22}.

A natural next step is to tackle this problem for a principally polarized abelian surface $A/\Q$. For any prime $\ell$, let $G_\ell = \GSp(4,\F_\ell)$. The $\ell$-torsion subgroup $A[\ell]$ is a $4$-dimensional vector space over $\F_\ell$, with a non-degenerate, Galois covariant, alternating pairing -- the Weil pairing. Thus the Galois action leads to a mod-$\ell$ Galois representation $\rhobar := \rhobar_{A,\ell} : G_\Q \rightarrow G_\ell$, such that composing with the similitude character $\chi_{\tsim} : G_\ell \rightarrow \F_\ell^\times$ gives the mod-$\ell$ cyclotomic character $\chi_\ell$. Given $A/\Q$ and $\ell$, the task is to determine the mod-$\ell$ Galois image $\im(\rhobar_{A,\ell}) \subseteq G_\ell$ up to conjugacy. 
\begin{remark}
    When $A_{\Qbar}$ has no extra endomorphisms, Serre's open image theorem implies that $\im(\rhobar_{A,\ell}) \neq G_\ell$ for finitely many primes $\ell$. Dieulefait gave an algorithm to determine a finite set containing the non-surjective primes \cite{Die02, BBKKMSV24}. Building on this, \cite{vBCCK24} gives a complete algorithm for computing the rational isogeny class of $A$.
\end{remark}
Since $\chi_\ell = \chi_{\tsim} \circ \rhobar$ is surjective, and complex conjugation is an order $2$ element $c \in G_{\Q}$ with $\chi_\ell(c) = -1$, there are constraints on a subgroup $H \subseteq G_\ell$ to be the mod-$\ell$ Galois image. To be precise, $H$ must be an eligible subgroup (see \cref{def:eligible}).
\begin{definition}
    \label{def:eligible}
    A subgroup $H$ of $G_\ell$ is said to be eligible if the restriction of $\chi_{\tsim}$ to $H$ is surjective, and there exists $x \in H$ of order $2$ such that $\chi_{\tsim}(x) = -1$.
\end{definition}

We develop a method to completely accomplish this task when $\ell = 3$. It works for all principally polarized abelian surfaces without any restriction on endomorphism type. It consists of a probabilistic Monte Carlo \cref{alg:GassmannClass} that determines the conjugacy class distribution of the image of $\rhobar_{A,3}$, followed by a deterministic \cref{alg:distinguish} that determines $\im(\rhobar_{A,3})$ up to $G_3$-conjugacy. We present the algorithms for Jacobians of genus $2$ curves, but they also apply for Weil restrictions of an elliptic curve over a quadratic field.
They also apply for curves over any number field $K$ not containing $\zeta_3$. If $\zeta_3 \in K$, the eligible subgroups are now those contained in $\Sp(4,\F_3)$ and \cref{alg:distinguish} will then require an analysis of these subgroups.

Magma \cite{Magma} code implementing the algorithm is available at the GitHub repository \cite{Chidambaram_Mod3Image}. The main intrinsic is called \texttt{mod3Galoisimage}, which takes a genus $2$ curve as input, and optionally a lower bound for probability $\epsilon$, and two prime bounds $B_1, B_2$ (see \cref{alg:GassmannClass}). Worked out examples can be found in the file \texttt{examples.m}.
We have used it to compute the mod-$3$ Galois images for all \numprint{66158} genus $2$ curves in the L-functions and Modular Forms Database LMFDB \cite{lmfdb}, and for a bigger dataset consisting of \numprint{487493} curves with $5$-smooth conductors. \cref{sec:ellis3} presents our main algorithms and \cref{sec:ellgt3} discusses the challenges for $\ell > 3$.

Any representation $\rhobar : G_\Q \rightarrow G_3$ with $\chi_\tsim \circ \rhobar = \chi_3$ is known to arise as the mod-$3$ Galois representation of infinitely many prinicipally polarized abelian surfaces over $\Q$ \cite{BCGP}. This is because the corresponding twist of the Siegel modular variety $\MA_2(3)$, which may not be rational over $\Q$ \cite{CC}, is nevertheless unirational over $\Q$. 
Consequently any eligible subgroup of $G_3$ arises as the mod-$3$ Galois image.
Up to conjugacy $G_3$ has $492$ subgroups; $330$ of them are not contained in $\Sp(4,\F_3)$ and $280$ of them are eligible subgroups.
This leads to the following question.
\begin{question}
    \label{question:realize}
    Can we \emph{explicitly} realize each of the $280$ eligible subgroups of $G_3$ as the mod-$3$ Galois image of some principally polarized abelian surface over $\Q$?
\end{question}
The work \cite{BF22} outlines a method to construct a genus $2$ curve starting from a mod-$3$ Galois representation $\rhobar$. It can be applied to realize eligible subgroups of $G_3$ of small order as mod-$3$ Galois images. 
\cref{sec:realize} discusses progress towards this.

\textbf{Notation:} Throughout this article, we use the LMFDB labeling scheme \cite{gpsgsp4labelknowl} for referring to subgroups of $G_\ell$. The label is a string \texttt{$\ell$.i.n}, where $\texttt{i}$ is the index in $G_\ell$, and $\texttt{n}$ is a counter giving a canonical ordering of all index-$i$ subgroups of $G_\ell$.

\textbf{Acknowledgement:}
The author was supported by Simons Foundation grant 550033. I thank Andrew Sutherland, David Roberts and Nils Bruin for valuable discussions, Noam Elkies for \cref{eg:index80subgroup}, and the referees whose inputs greatly improved the article, and simplified \cref{eg:verysmallimage}.

\section{Computing the mod-$3$ Galois image}
\label{sec:ellis3}
Suppose $A$ is given as the Jacobian of a genus $2$ curve $C/\Q : y^2 = f(x)$ with $\deg(f) = 5$ or $6$. Let $S$ denote the set of primes of bad reduction for $A$. 
We want to determine $\im(\rhobar_{A,3}) \subseteq G_3$ up to conjugacy. This group is abstractly isomorphic to the Galois group of the $3$-torsion field. Although the latter can be computed from a $3$-division polynomial, it is not enough to determine $\im(\rhobar_{A,3})$ up to $G_3$-conjugacy. The difficulty comes from the presence of outer automorphisms.
\cref{ellis2} discusses this issue in the simpler case $\ell = 2$.

\begin{remark}
\label{ellis2}
For $\ell = 2$, since the $2$-torsion field of $A$ is the splitting field of $f(x)$, the mod-$2$ Galois image is abstractly isomorphic to the Galois group $\Gal(f) \subseteq S_6 \simeq G_2$. But the identification $S_6 \simeq G_2$ must be chosen carefully since $S_6$ has an outer automorphism. This is done, for instance, by noting that pairs of Weierstrass points on $C$ correspond to points in $A[2]$, and matching the two conjugacy classes of order-$48$ subgroups correctly: the subgroup of $S_6$ that has orbits of size $2$ and $4$ must match with the subgroup of $G_2$ having a fixed point, while the other transitive subgroup of $S_6$ must match with the one having no fixed points.
\end{remark}

\subsection{Computing the Gassmann-equivalence class.}
When computing any Galois group, the Chebotarev density theorem is a very useful tool. Suppose 
the Galois group is a subgroup of a group $G$. 
Then one computes the conjugacy class of $\Frob_p$ for all unramified primes $p \leq X$. This sampled frequency is compared against the conjugacy class distributions of all subgroups of $G$, and ideally this pins down the Galois group with high certainty, since Chebotarev guarantees that the distribution of $\Frob_p$ for $p \leq X$ matches that of the Galois group as $X \rightarrow \infty$.

This general approach yields the probabilistic method of \cref{alg:GassmannClass}, wherein we prune out subgroups whose conjugacy class distribution is unlikely to yield the sampled Frobenius distribution. Since computing the Frobenius conjugacy class in $G_3$ is intensive, we do it only when needed. 
Instead we focus on two essential (and efficiently computable) features, together called the signature of a conjugacy class in $G_3$: the characteristic polynomial, and the dimension of the $1$-eigenspace.
\begin{definition}
    For a prime $p$ of good reduction for $A$, the Frobenius signature of $A$ at $p$ is defined to be the tuple $\langle L_p(A,t) \pmod 3, \dim_{\F_p} A[3](\F_p) \rangle$, where $L_p(A,t)$ is the Euler factor of $A$ at $p$. Note that $\det(I - \rhobar_{A,3}(\Frob_p)t) \equiv L_p(A,t) \pmod 3$.
\end{definition}
Although \cref{alg:GassmannClass} goes far, it cannot always determine the mod-$3$ Galois image due to the existence of non-conjugate but Gassmann-equivalent \cite[Def 2.8]{Sut21} subgroups, i.e., subgroups giving rise to the same conjugacy class distribution. The $280$ eligible subgroups of $G_3$ give rise to $230$ distinct conjugacy class distributions. These are $38$ pairs, $3$ triples and $2$ quadruples of Gassmann-equivalent subgroups.

\begin{algorithm}
    \label{alg:GassmannClass}
    {\em Input:}
    \begin{itemize}
        \item a genus 2 curve $C$ over $\Q$ such that $A = \Jac(C)$
        \item a positive real number $\epsilon$ close to $0$; two integers $B_1 \geq B_2 \geq 1$.
        \item a pre-computed list $L$ of the $280$ eligible subgroups of $G_3$ along with their signature distributions and conjugacy class distributions.
    \end{itemize}
    {\em Output:} the Gassmann-equivalence class of subgroups of $G_3$ containing $\im (\rhobar_{A,3})$.
    \begin{enumerate}[label={\bf Step \arabic*},labelwidth=\widthof{{\bf Step 2.}},leftmargin=!]
    \item \label{step:sampleFrobsign} For each prime $p \not \in S \cup \{3\}, p \leq B_1$, compute the Frobenius signature.
    \item Using the pre-computed list of signature distributions, apply Bayes' rule to calculate for each $H \in L$ the probablity that $\im(\rhobar_{A,3}) = H$ given the sampled Frobenius sign distribution from \ref{step:sampleFrobsign}.
    \item Remove those $H \in L$ whose probability is smaller than $\epsilon$. If the remaining subgroups form a unique Gassmann-equivalence class, return it.
    \item  \label{step:sampleFrob} Otherwise, for each prime $p \not \in S \cup \{3\}, p \leq B_2$, compute the conjugacy class of $\rhobar_{A,3}(\Frob_p) \in G_3$ by computing $A[3](\Fbar_p)$ as a subgroup of $A(\F_{p^k})$ for some $k \geq 1$, and constructing a symplectic basis of $A[3]$.
    \item  Using the pre-computed list of conjugacy class distributions, apply Bayes' rule to calculate for each remaining $H \in L$ the probablity that $\im(\rhobar_{A,3}) = H$ given the sampled Frobenius conjugacy class distribution from \ref{step:sampleFrob}.
    \item Remove those $H$ whose probability is smaller than $\epsilon$. If the remaining subgroups form a unique Gassmann-equivalence class, return it.
    \end{enumerate}
\end{algorithm}

\begin{prop}
\label{prop:GassmannClass}
For any $0 < \epsilon < \frac{1}{4}$,
there exists $B$ such that \cref{alg:GassmannClass} with $B_i \geq B$ returns the Gassmann-equivalence class containing $\im(\rhobar_{A,3})$.
\end{prop}
\begin{proof}
    By the Chebotarev density theorem, the sampled Frobenius distribution converges to the conjugacy class distribution of the Gassmann-equivalence class containing $\im(\rhobar_{A,3})$, as $B_2 \rightarrow \infty$. The proposition follows immediately.
\end{proof}

\begin{example}
    The largest subgroup of $G_3$ that fails the local-global principle for stabilizing a maximal isotropic subspace of $\F_3^4$ is $H = $ \href{https://www.lmfdb.org/knowledge/show/gsp4.subgroup_data?label=3.1080.4}{3.1080.4}, i.e., every element of $H$ stabilizes some $2$-dimensional isotropic plane, but $H$ does not. Consider the genus $2$ curve with LMFDB label \href{https://www.lmfdb.org/Genus2Curve/Q/25600/f/512000/1}{25600.f.512000.1}. \cref{alg:GassmannClass} returns that the mod-$3$ Galois image for this curve is $H$. So the Jacobian of $C$ presents an example where the local-global principle for the existence of isogenies fails.
\end{example}

\begin{example}
\label{eg:index80subgroup}
Let $H$ denote the stabilizer in $G_3$ of an isotropic plane in $\F_3^4$. It is the group \href{https://www.lmfdb.org/knowledge/show/gsp4.subgroup_data?label=3.40.2}{3.40.2}, and has three index-$2$ subgroups: $H \cap \Sp(4,\F_3)$ and the two eligible subgroups \href{https://www.lmfdb.org/knowledge/show/gsp4.subgroup_data?label=3.80.3}{3.80.3} and \href{https://www.lmfdb.org/knowledge/show/gsp4.subgroup_data?label=3.80.4}{3.80.4}. We note that \href{https://www.lmfdb.org/knowledge/show/gsp4.subgroup_data?label=3.80.3}{3.80.3} is the largest subgroup of $G_3$ that does not occur as the mod-$3$ Galois image for any genus $2$ curve in the LMFDB dataset. Consider the moduli space $X/\Q$ parametrizing $(3,3)$-isogenies of principally polarized abelian surfaces (analog of the modular curve $X_0(3)$), i.e., the moduli space of abelian surfaces $A$ with $\im(\rhobar_{A,3}) \subseteq \href{https://www.lmfdb.org/knowledge/show/gsp4.subgroup_data?label=3.40.2}{3.40.2}$. The three index-$2$ subgroups correspond to three degree-$2$ covers. If $\Q(X)$ denotes the function field of $X$, the function fields of the three covers are respectively $\Q(\sqrt{-3})(X), \Q(X)(\sqrt{-3f})$ and $\Q(X)(\sqrt{f})$ for some $f \in \Q(X)$. Using a birational model of $X$ and the curves in LMFDB whose mod-$3$ Galois image is \href{https://www.lmfdb.org/knowledge/show/gsp4.subgroup_data?label=3.80.4}{3.80.4}, Noam Elkies guessed the function $f$. Then a search for rational points on $X$ where $-3f$ is a square, yields the curve $C: y^2 = -27x^6 + 54x^5 - 693x^4 + 1278x^3 - 543x^2 - 60x - 16$ with conductor $3^2~7^4~13~43^2$. \cref{alg:GassmannClass} verifies that its mod-$3$ Galois image is indeed \href{https://www.lmfdb.org/knowledge/show/gsp4.subgroup_data?label=3.80.3}{3.80.3}.
\end{example}

\begin{remark}
    The output of \cref{alg:GassmannClass} is rigorous only when it is the whole group $G_3$. An effective Chebotarev density theorem \cite{LO77} can be used to make \cref{alg:GassmannClass} completely rigorous, if we sample a very large number of primes. Let $\delta$ be the minimum of the $L^{\infty}$ distance between any two conjugacy class distributions of eligible subgroups of $G_3$. Bounding the error term in Chebotarev density by $\delta/2$ pins down a unique Gassmann-equivalence class. By \cite[Thm 1.1]{LO77}, which assumes GRH for zeta functions of number fields, this can be accomplished by sampling primes up to $x$, if $x > O(\exp(k\delta^{-1}))$ for some explicit constant $k$ and 
    $x$ is greater than the root discriminant of $\Q(A[3])|\Q$. Note that the local root discriminant of $\Q(A[3])|\Q$ at a bad prime $p \ne 2,3$ is less than or equal to $p$. 
\end{remark}

\begin{remark}
    If the abstract isomorphism class of $\Gal(\Q(A[3])|\Q)$ is known, the list $L$ in \cref{alg:GassmannClass} has to contain only those eligible subgroups of $G_3$ in this isomorphism class. This makes the output of \cref{alg:GassmannClass} rigorous in more cases. But whenever this class contains non-conjugate subgroups of $G_3$ meeting the same set of $G_3$-conjugacy classes, the probabilistic nature of \cref{alg:GassmannClass} is unavoidable.
\end{remark}

\subsection{Distinguishing Gassmann-equivalent subgroups}
We begin by noting a canonical example of non-conjugate Gassmann-equivalent subgroups of $G_\ell$.
\begin{example}
    Let 
$J_2 = \begin{bmatrix} & 1\\-1 & \end{bmatrix}$, $\calJ_2 = \begin{bmatrix} & 1\\1 & \end{bmatrix}$. Suppose $\calJ'_2 = \begin{bmatrix} & \calJ_2 \\ -\calJ_2 &  \end{bmatrix}$
represents the standard alternating form on the column space $\F_\ell^4$, that is preserved by $G_\ell$ up to scalars. Let $H_1 \subseteq G_\ell$ be the stabilizer of a vector in $\F_\ell^4$, so $H_1 = \left\{\begin{bmatrix} 1 & B & c\\  & A & D\\  & & x \end{bmatrix} : D = A J_2 B^t, \det(A) = x \right\}$. Let $H_2 \subseteq G_\ell$ be the transpose of $H_1$. Then $H_1$ and $H_2$ are clearly non-conjugate, but they are Gassmann-equivalent \cite[Prop 2.6]{Sut21} since transposing gives a $G_\ell$-conjugacy preserving bijection $H_1 \leftrightarrow H_2$.
\end{example}

In \cite{Sut16}, while computing the mod-$\ell$ Galois image $\im(\rhobar_{E,\ell})$ of an elliptic curve $E/\Q$, Sutherland tackles the problem of distinguishing Gassmann-equivalent subgroups of $\GL(2,\F_\ell)$ by computing the degree of the minimal number field over which $E$ acquires $\ell$-torsion points. For example, if $\im(\rhobar_{E,\ell}) = \begin{bmatrix} 1 & *\\0 & * \end{bmatrix}$, then $E(\Q)$ must have an $\ell$-torsion point, whereas if $\im(\rhobar_{E,\ell}) = \begin{bmatrix} * & *\\0 & 1 \end{bmatrix}$, then the minimal degree of a number field $K$ such that $E(K)$ has an $\ell$-torsion point is $\ell-1$.

We develop this idea further in \cref{alg:distinguish} to distinguish any two Gassmann-equivalent subgroups of $G_3$ that are not conjugate in $\GL(4,\F_3)$. There are exactly $5$ Gassmann-equivalence classes of size $2$, listed in \cref{table:GassmannEquivalentGLConjugate}, containing $\GL(4,\F_3)$-conjugate subgroups.
These are indistinguishable by such methods. Since these groups are small (order $\leq 32$), we distinguish them in \cref{alg:distinguish} by literally constructing a symplectic basis of $A[3]$ over the $3$-torsion field $K$.

\begin{table}[h]
\caption{Non-conjugate Gassmann-equivalent subgroups of $\GSp(4,\F_3)$ that are conjugate in $\GL(4,\F_3)$.}
\label{table:GassmannEquivalentGLConjugate}
\begin{tabular}{|c|c|p{6.3cm}|}
\hline
$|H|$ & Label & Entries in generators read horizontally \\
\hline
32 & \href{https://www.lmfdb.org/knowledge/show/gsp4.subgroup_data?label=3.3240.6}{3.3240.6} & [ 0, 0, 1, 2, 2, 2, 0, 1, 1, 1, 1, 0, 2, 1, 1, 0 ]\newline[ 2, 1, 0, 1, 2, 2, 0, 1, 1, 1, 2, 1, 2, 1, 1, 0 ]\newline[ 1, 1, 2, 1, 2, 0, 0, 2, 1, 2, 0, 2, 2, 1, 1, 2 ]\\
32 & \href{https://www.lmfdb.org/knowledge/show/gsp4.subgroup_data?label=3.3240.7}{3.3240.7} & [ 2, 0, 0, 0, 1, 0, 2, 0, 2, 2, 0, 0, 1, 2, 2, 1 ]\newline[ 1, 0, 2, 1, 0, 1, 1, 2, 0, 0, 2, 0, 0, 0, 0, 2 ]\newline[ 2, 1, 2, 0, 2, 2, 0, 2, 1, 1, 1, 2, 2, 1, 1, 1 ]\\
\hline
16 & \href{https://www.lmfdb.org/knowledge/show/gsp4.subgroup_data?label=3.6480.16}{3.6480.16} & [ 2, 0, 0, 1, 0, 1, 0, 0, 0, 1, 2, 0, 0, 0, 0, 1 ]\newline[ 0, 2, 1, 0, 1, 2, 1, 1, 2, 0, 0, 1, 1, 2, 2, 2 ]\\
16 & \href{https://www.lmfdb.org/knowledge/show/gsp4.subgroup_data?label=3.6480.3}{3.6480.3} & [ 0, 2, 1, 0, 2, 1, 2, 1, 1, 1, 2, 1, 2, 1, 1, 0 ]\newline[ 2, 2, 1, 0, 0, 1, 1, 1, 0, 0, 2, 1, 0, 0, 0, 1 ]\\
\hline
16 & \href{https://www.lmfdb.org/knowledge/show/gsp4.subgroup_data?label=3.6480.13}{3.6480.13} & [ 1, 2, 1, 0, 1, 1, 0, 1, 2, 2, 2, 1, 1, 2, 2, 2 ]\newline[ 1, 0, 2, 1, 0, 1, 1, 2, 0, 0, 2, 0, 0, 0, 0, 2 ]\\
16 & \href{https://www.lmfdb.org/knowledge/show/gsp4.subgroup_data?label=3.6480.17}{3.6480.17} & [ 0, 1, 2, 2, 2, 2, 2, 2, 1, 1, 1, 2, 0, 1, 1, 0 ]\newline[ 2, 0, 2, 1, 2, 1, 2, 2, 1, 1, 2, 0, 2, 1, 1, 1 ]\\
\hline
16 & \href{https://www.lmfdb.org/knowledge/show/gsp4.subgroup_data?label=3.6480.14}{3.6480.14} & [ 2, 1, 0, 1, 0, 2, 0, 0, 2, 2, 1, 2, 0, 2, 0, 1 ]\newline[ 2, 0, 1, 2, 2, 1, 2, 1, 0, 0, 2, 0, 0, 0, 1, 1 ]\newline[ 1, 1, 1, 0, 2, 2, 0, 2, 2, 0, 2, 1, 0, 1, 2, 1 ]\\
16 & \href{https://www.lmfdb.org/knowledge/show/gsp4.subgroup_data?label=3.6480.15}{3.6480.15} & [ 1, 2, 0, 2, 0, 1, 0, 0, 1, 1, 2, 1, 0, 1, 0, 2 ]\newline[ 2, 0, 1, 2, 1, 0, 0, 1, 2, 2, 0, 0, 1, 2, 2, 1 ]\newline[ 2, 0, 1, 1, 1, 1, 1, 1, 2, 1, 2, 0, 1, 2, 2, 1 ]\\
\hline
8 & \href{https://www.lmfdb.org/knowledge/show/gsp4.subgroup_data?label=3.12960.5}{3.12960.5} & [ 2, 0, 0, 1, 0, 1, 0, 0, 0, 1, 2, 0, 0, 0, 0, 1 ]\newline[ 1, 0, 2, 1, 2, 0, 0, 2, 1, 1, 0, 0, 2, 1, 1, 2 ]\\
8 & \href{https://www.lmfdb.org/knowledge/show/gsp4.subgroup_data?label=3.12960.11}{3.12960.11} & [ 2, 2, 1, 0, 1, 1, 1, 1, 2, 0, 2, 1, 1, 2, 2, 1 ]\newline[ 0, 2, 1, 1, 1, 1, 1, 1, 2, 2, 2, 1, 0, 2, 2, 0 ]\\
\hline
\end{tabular}
\end{table}

We set up some notation before describing the algorithm. Let $\calK = A/\{\pm\}$ denote the Kummer surface of $A$, and let $\calX$ denote the image of $A[3]$ under the quotient map. Explicit formulas are known for doubling on $\calK$ \cite[Thm 3.4.1]{cassels_flynn_1996}. Using these, equations for the $3$-torsion locus $\calX \subset \calK$ can be computed, as in \cite{CCG}, by writing $2 P = - P$ in terms of the coordinates of an arbitrary point $P \in \calK$. Although taking successive resultants of these equations can yield a $3$-division polynomial, knowing a general $3$-division polynomial at the outset makes \ref{step:bruteforce} and \ref{step:dimthreetorsovernfield} of \cref{alg:distinguish} faster. We use such a polynomial given to us by David Roberts.
\begin{algorithm}
    \label{alg:distinguish}
    {\em Input:}
    \begin{itemize}
        \item a genus 2 curve $C$ over $\Q$ such that $A = \Jac(C)$
        \item a Gassmann-equivalence class $L$ of subgroups of $G_3$ containing $\im (\rhobar_{A,3})$.
    \end{itemize}
    {\em Output:} the image of $\rhobar_{A,3}$ up to conjugacy in $G_3$.
    \begin{enumerate}[label={\bf Step \arabic*},labelwidth=\widthof{{\bf Step 2.}},leftmargin=!]
    \item \label{step:bruteforce} If $L$ is a class given in \cref{table:GassmannEquivalentGLConjugate}, compute the three-torsion field $K = \Q(A[3])$ and its automorphism group $\Gal(K|\Q)$. Find all $41$ geometric points on the three-torsion locus $\calX$, i.e., compute $\calX(\Qbar) = \calX(K)$, and lift points to find a basis of $A[3]$. Fix a symplectic basis by working in some residue field. Compute matrices with respect to this basis corresponding to the action of the generators of $\Gal(K|\Q)$, thus determining $\im(\rhobar)$.
    \item \label{step:maxptsoverext} Otherwise, for each $H \in L$, compute $\dim (\F_3^4)^{H_0}$ where $H_0 = H \cap \Sp(4,\F_3)$. Also compute $\max\limits_{[H:H_1] = d} \dim (\F_3^4)^{H_1}$ for each $d \in D = \{1,2,3,6,8,12\}$.
    \item \label{step:dimthreetorsovernfield} Compute $\calX(\Q(\zeta_3))$, and the points in $\calX$ whose degree belongs to the set $D$. Lift them to $3$-torsion points on $A$, and thus compute $\dim_{\F_3} (A[3](\Q(\zeta_3)))$, and $\max\limits_{[K:\Q] = d} \dim_{\F_3}(A[3](K))$ for each $d \in D$.
    \item If there is a unique group $H \in L$, whose data computed in \ref{step:maxptsoverext} matches the data computed in \ref{step:dimthreetorsovernfield}, return it.
    \end{enumerate}
\end{algorithm}

\begin{theorem}
\label{thm:distinguish}
Given any genus 2 curve $C/\Q$ with $A = \Jac(C)$, and the Gassmann-equivalence class $L$ containing $\im(\rhobar_{A,3})$, \cref{alg:distinguish} returns $\im(\rhobar_{A,3})$.
\end{theorem}
\begin{proof}
Suppose $L$ is a Gassmann-equivalence class of size $> 1$ not listed in \cref{table:GassmannEquivalentGLConjugate}. If $L$ appears in Tables \ref{table:distinguishGLnonconjugate_smalldegree} or \ref{table:distinguishGLnonconjugate_largedegree}, the corresponding set of indices $d$ shown in the table distinguish all subgroups appearing in $L$.
Otherwise, the subgroups in $L$ are distinguished by the dimension of their fixed spaces.
\end{proof}

\begin{example}
    Let $H$ denote the stabilizer in $G_3$ of an line in $\F_3^4$. It is the group \href{https://www.lmfdb.org/knowledge/show/gsp4.subgroup_data?label=3.40.1}{3.40.1}, and two of its index-$2$ subgroups \{\href{https://www.lmfdb.org/knowledge/show/gsp4.subgroup_data?label=3.80.1}{3.80.1}, \href{https://www.lmfdb.org/knowledge/show/gsp4.subgroup_data?label=3.80.2}{3.80.2}\} form a Gassmann-equivalence class. While \href{https://www.lmfdb.org/knowledge/show/gsp4.subgroup_data?label=3.80.1}{3.80.1} is the stabilizer of a point, \href{https://www.lmfdb.org/knowledge/show/gsp4.subgroup_data?label=3.80.2}{3.80.2} acts trivially on a $1$-dimensional quotient of $\F_3^4$.
    Consider the genus $2$ curve $C$ with LMFDB label \href{https://www.lmfdb.org/Genus2Curve/Q/277/a/277/2}{277.a.277.2}. Running \cref{alg:GassmannClass} on $C$ returns the correct Gassmann-equivalence class. Further running \cref{alg:distinguish} returns that the mod-$3$ Galois image is the group \href{https://www.lmfdb.org/knowledge/show/gsp4.subgroup_data?label=3.80.2}{3.80.2}. So the local-global principle for the existence of a $3$-torsion point fails for $\Jac(C)$. This is consistent with Mordell-Weil group $\Jac(C)(\Q) \simeq \Z/5$.
\end{example}

\begin{table}[h]
\caption{Gassmann-equivalent subgroups distinguished by $\dim (\F_3^4)^{H_0}$ and $\max\limits_{[H:H_1] = d} \dim (\F_3^4)^{H_1}$ for $d \leq 3$.}
\label{table:distinguishGLnonconjugate_smalldegree}
\begin{tabular}{|c|c|p{6.3cm}|c|c|c|c|}
\hline
$|H|$ & Label & Entries in generators read horizontally & $H_0$ & d=1 & 2 & 3\\
\hline
324 & \href{https://www.lmfdb.org/knowledge/show/gsp4.subgroup_data?label=3.320.1}{3.320.1}& [ 2, 2, 2, 0, 2, 1, 1, 0, 2, 0, 1, 2, 2, 0, 2, 2 ]\newline[ 1, 0, 2, 0, 2, 1, 1, 2, 2, 2, 0, 1, 2, 1, 0, 1 ] & 1&1&1&1\\
& \href{https://www.lmfdb.org/knowledge/show/gsp4.subgroup_data?label=3.320.2}{3.320.2}& [ 1, 1, 1, 2, 0, 0, 2, 1, 2, 0, 0, 1, 1, 1, 2, 2 ]\newline[ 0, 1, 0, 1, 2, 1, 2, 1, 1, 2, 1, 0, 0, 2, 2, 1 ] & 1&0&1&0\\
& \href{https://www.lmfdb.org/knowledge/show/gsp4.subgroup_data?label=3.320.5}{3.320.5}& [ 1, 2, 1, 1, 0, 1, 2, 0, 0, 0, 2, 2, 0, 0, 0, 2 ]\newline[ 2, 2, 0, 0, 2, 1, 2, 1, 0, 0, 1, 2, 1, 1, 2, 2 ] & 0&0&1&1\\
& \href{https://www.lmfdb.org/knowledge/show/gsp4.subgroup_data?label=3.320.6}{3.320.6}& [ 1, 2, 1, 1, 1, 0, 2, 0, 0, 2, 1, 0, 2, 2, 1, 1 ]\newline[ 1, 0, 2, 1, 1, 1, 0, 0, 1, 0, 0, 1, 1, 1, 2, 1 ] & 0&0&1&0\\
\hline
162 & \href{https://www.lmfdb.org/knowledge/show/gsp4.subgroup_data?label=3.640.1}{3.640.1}& [ 2, 2, 0, 0, 2, 0, 1, 2, 0, 2, 0, 0, 1, 2, 0, 2 ]\newline[ 0, 0, 1, 2, 1, 1, 0, 1, 0, 0, 2, 0, 2, 0, 2, 0 ] & 1&1&1&2\\
& \href{https://www.lmfdb.org/knowledge/show/gsp4.subgroup_data?label=3.640.2}{3.640.2}& [ 0, 1, 0, 1, 0, 1, 0, 2, 2, 2, 2, 1, 1, 1, 2, 1 ]\newline[ 1, 1, 1, 0, 0, 0, 2, 1, 0, 1, 2, 2, 0, 0, 0, 1 ]\newline[ 2, 2, 2, 1, 0, 0, 1, 2, 0, 1, 0, 1, 0, 0, 0, 1 ] & 1&1&1&1\\
& \href{https://www.lmfdb.org/knowledge/show/gsp4.subgroup_data?label=3.640.3}{3.640.3}& [ 2, 2, 0, 2, 1, 1, 1, 0, 1, 2, 1, 0, 1, 2, 0, 2 ]\newline[ 2, 2, 0, 2, 0, 2, 1, 2, 1, 2, 1, 0, 2, 0, 2, 2 ] & 1&0&1&0\\
& \href{https://www.lmfdb.org/knowledge/show/gsp4.subgroup_data?label=3.640.4}{3.640.4}& [ 0, 0, 2, 2, 0, 0, 2, 1, 2, 1, 1, 2, 1, 1, 2, 2 ]\newline[ 2, 1, 2, 1, 2, 2, 0, 0, 0, 1, 2, 0, 1, 1, 2, 1 ]\newline[ 2, 1, 2, 2, 2, 2, 0, 1, 0, 0, 1, 1, 1, 2, 0, 2 ] & 1&0&1&1\\
\hline
36 & \href{https://www.lmfdb.org/knowledge/show/gsp4.subgroup_data?label=3.2880.13}{3.2880.13}& [ 0, 0, 1, 2, 2, 2, 2, 1, 1, 1, 1, 0, 0, 1, 1, 0 ]\newline[ 2, 0, 0, 0, 0, 2, 0, 0, 1, 1, 1, 0, 2, 1, 0, 1 ]\newline[ 2, 1, 1, 1, 2, 2, 2, 1, 2, 1, 2, 2, 2, 2, 1, 2 ] & 0&0&2&0\\
& \href{https://www.lmfdb.org/knowledge/show/gsp4.subgroup_data?label=3.2880.17}{3.2880.17}& [ 1, 2, 1, 0, 1, 1, 0, 1, 0, 1, 0, 1, 2, 0, 2, 0 ]\newline[ 0, 1, 0, 2, 1, 2, 2, 0, 0, 2, 0, 2, 2, 0, 2, 2 ]\newline[ 2, 2, 0, 2, 2, 1, 2, 0, 0, 1, 2, 1, 1, 0, 1, 1 ] & 0&0&1&0\\
\hline
18 & \href{https://www.lmfdb.org/knowledge/show/gsp4.subgroup_data?label=3.5760.2}{3.5760.2}& [ 1, 1, 1, 2, 0, 0, 2, 1, 0, 2, 0, 2, 0, 0, 0, 2 ]\newline[ 2, 1, 2, 0, 0, 0, 2, 2, 1, 0, 2, 2, 2, 1, 0, 0 ] & 2&1&2&2\\
& \href{https://www.lmfdb.org/knowledge/show/gsp4.subgroup_data?label=3.5760.5}{3.5760.5} & [ 2, 1, 2, 0, 0, 0, 2, 2, 1, 0, 2, 2, 2, 1, 0, 0 ]\newline[ 1, 2, 1, 2, 1, 0, 2, 1, 0, 1, 0, 1, 2, 0, 2, 2 ] & 2&1&2&1\\
\hline
12 & \href{https://www.lmfdb.org/knowledge/show/gsp4.subgroup_data?label=3.8640.2}{3.8640.2} & [ 1, 0, 0, 0, 0, 1, 0, 0, 1, 1, 2, 0, 2, 1, 0, 2 ]\newline[ 1, 2, 2, 0, 2, 2, 1, 2, 2, 1, 2, 1, 0, 2, 1, 0 ] & 0&0&2&0\\
& \href{https://www.lmfdb.org/knowledge/show/gsp4.subgroup_data?label=3.8640.4}{3.8640.4} & [ 0, 2, 1, 1, 1, 0, 2, 1, 1, 2, 0, 1, 1, 1, 2, 0 ]\newline[ 2, 1, 0, 0, 0, 1, 0, 0, 0, 0, 2, 2, 0, 0, 0, 1 ] & 0&0&1&0\\
\hline
12 & \href{https://www.lmfdb.org/knowledge/show/gsp4.subgroup_data?label=3.8640.12}{3.8640.12} & [ 1, 1, 0, 0, 0, 2, 0, 0, 2, 0, 1, 2, 1, 2, 0, 2 ]\newline[ 2, 0, 1, 1, 0, 2, 1, 1, 2, 1, 1, 0, 1, 2, 0, 1 ] & 0&0&2&0\\
& \href{https://www.lmfdb.org/knowledge/show/gsp4.subgroup_data?label=3.8640.13}{3.8640.13} & [ 2, 1, 1, 1, 0, 0, 1, 1, 1, 1, 1, 2, 2, 1, 0, 2 ]\newline[ 2, 0, 2, 1, 2, 1, 0, 2, 0, 0, 2, 0, 0, 0, 1, 1 ] & 0&0&1&0\\
\hline
\end{tabular}
\end{table}

\begin{table}[h]
\caption{Gassmann-equivalent subgroups distinguished by 
$\max\limits_{[H:H_1] = d} \dim (\F_3^4)^{H_1}$
for $d = 6, 8, 12$.}
\label{table:distinguishGLnonconjugate_largedegree}
\begin{tabular}{|c|c|p{6.3cm}|c|c|c|}
\hline
$|H|$ & Label & Entries in generators read horizontally & d=6 & 8 & 12 \\
\hline
432 & \href{https://www.lmfdb.org/knowledge/show/gsp4.subgroup_data?label=3.240.6}{3.240.6} & [ 0, 0, 0, 2, 0, 1, 2, 2, 0, 0, 2, 0, 2, 0, 1, 0 ]\newline[ 1, 0, 0, 2, 0, 2, 1, 0, 0, 0, 2, 0, 2, 0, 0, 2 ] & 0&1&0\\
& \href{https://www.lmfdb.org/knowledge/show/gsp4.subgroup_data?label=3.240.7}{3.240.7} & [ 2, 0, 1, 1, 0, 1, 1, 2, 0, 0, 2, 0, 2, 0, 1, 2 ]\newline[ 0, 0, 2, 2, 1, 2, 0, 2, 0, 0, 1, 0, 2, 0, 2, 0 ] & 0&1&1\\
\hline
324 & \href{https://www.lmfdb.org/knowledge/show/gsp4.subgroup_data?label=3.320.3}{3.320.3} & [ 1, 1, 0, 2, 1, 0, 2, 0, 0, 2, 1, 2, 2, 0, 2, 0 ]\newline[ 2, 0, 0, 1, 2, 1, 1, 1, 0, 1, 0, 0, 1, 2, 0, 0 ] & 2&-&2\\
& \href{https://www.lmfdb.org/knowledge/show/gsp4.subgroup_data?label=3.320.4}{3.320.4} & [ 1, 2, 2, 1, 2, 1, 2, 0, 2, 0, 0, 2, 2, 0, 2, 2 ]\newline[ 1, 1, 1, 2, 2, 1, 2, 0, 2, 0, 0, 1, 2, 1, 0, 1 ]\newline[ 2, 1, 1, 1, 1, 2, 1, 1, 1, 1, 1, 2, 1, 1, 2, 1 ] & 1&-&2\\
\hline
48 & \href{https://www.lmfdb.org/knowledge/show/gsp4.subgroup_data?label=3.2160.9}{3.2160.9} & [ 1, 1, 1, 1, 0, 2, 1, 1, 0, 0, 1, 2, 0, 0, 0, 2 ]\newline[ 0, 2, 0, 1, 2, 1, 1, 0, 2, 0, 1, 0, 2, 1, 0, 2 ] & 1&1&2\\
& \href{https://www.lmfdb.org/knowledge/show/gsp4.subgroup_data?label=3.2160.10}{3.2160.10} & [ 0, 2, 1, 0, 1, 0, 0, 1, 0, 1, 2, 2, 2, 1, 0, 2 ]\newline[ 2, 2, 2, 0, 0, 2, 0, 2, 1, 0, 0, 0, 2, 0, 2, 0 ] & 1&2&2\\
\hline
\end{tabular}
\end{table}

\section{Challenges in computing mod-$\ell$ Galois image for $\ell > 3$}
\label{sec:ellgt3}
For any prime $\ell > 3$, \cref{alg:GassmannClass} still works in principle to determine the Gassmann-equivalence class of $\im(\rhobar_{A,\ell})$. The precomputation of all eligible subgroups of $G_\ell$ becomes hard for large $\ell$, but for practical purposes, it is enough to compute only those subgroups whose index is less than a well-chosen bound.

But the problem of distinguishing Gassmann-equivalent subgroups gets considerably harder. Firstly, there are many more Gassmann-equivalent subgroups, for eg., 
$G_5$ has $1125$ eligible subgroups, but they give rise to only $773$ distinct conjugacy class distributions.
Secondly, getting information about $A[\ell]$ as in \ref{step:dimthreetorsovernfield} of \cref{alg:distinguish} is a lot more challenging. For eg., computing equations for the $\ell$-torsion locus on the Kummer surface seems hard already for $\ell = 7$.
Thirdly, there are fairly big Gassmann-equivalent subgroups of $G_\ell$ that are conjugate inside $\GL(4,\F_\ell)$. For eg., there is such a pair of subgroups \{\href{https://www.lmfdb.org/knowledge/show/gsp4.subgroup_data?label=5.48750.5}{5.48750.5}, \href{https://www.lmfdb.org/knowledge/show/gsp4.subgroup_data?label=5.48750.6}{5.48750.6}\} of $G_5$ of order $768$.

We can overcome some of these difficulties by taking endomorphisms into account. 
They constrain Galois action leading to fewer candidates for the Galois image. Also, for fixed endomorphism data, having a very small image is rare. An upcoming work of the author will fully address the case of no extra endomorphisms.

\section{Towards explicitly realizing all mod-$3$ Galois images}
\label{sec:realize}

As mentioned in the introduction, \cite{BF22} yields in theory a way to explicitly realize any of the $280$ eligible subgroups $H$ of $G_3$ as the mod-$3$ Galois image of the Jacobian of a genus $2$ curve over $\Q$. Given such a subgroup, one first obtains a number field $K$ properly solving the embedding problem
\begin{equation}
\label{eq:embeddingproblem}
\begin{tikzcd}
    & & & G_{\Q} \arrow{d}{\chi_3} \arrow[dashed]{dl}[swap]{\rhobar}& \\
    0 \arrow{r} & H \cap \Sp(4,\F_3) \arrow{r} & H \arrow{r} & \Z/2 \arrow{r} & 0
\end{tikzcd}
\end{equation}
i.e., $\rhobar$ is surjective and $K$ is the fixed field of $\ker (\rhobar)$. Since $G_3$ is a split extension of $\Z/2$ by $\Sp(4,\F_3)$, this embedding problem is always solvable and such a number field $K$ exists. Then \cite{BF22} gives a recipe to construct the corresponding twisted Burkhardt quartic threefold $B$, and associate to the rational points on $B$ certain genus $2$ curves. Finally, we make a suitable quadratic twist to
ensure that the mod-$3$ Galois image of the Jacobian is exactly equal to $H$. 
Here is an illustrative example to conclude.
\begin{example}
\label{eg:verysmallimage}
    Consider the subgroup $H \subset G_3$ with label \href{https://www.lmfdb.org/knowledge/show/gsp4.subgroup_data?label=3.12960.9}{3.12960.9}. It is isomorphic to $(\Z/2)^3$. We use the \texttt{BurkhardtModel} intrinsic with input $K = \Q(\zeta_{24})$ and an isomorphism $\rhobar: \Gal(K|\Q) \simeq H$ solving \cref{eq:embeddingproblem}, to find a model for the corresponding twisted Burkhardt quartic threefold $B$. We then follow the recipe in \cite{BF22} with a random rational point on $B$. This produces the quadratic twist by $2$ of the curve $C: y^2 = x(x^4 - 6840 x^2 + 456976)$,
    and \cref{alg:GassmannClass} verifies that its mod-$3$ Galois image is indeed $H$. The Jacobian of $C$ is isogenous to the product of the elliptic curve \href{https://www.lmfdb.org/EllipticCurve/Q/14/a/3}{14.a3} and its twist by $-1$. It has conductor $2^5 7^2 = 1568$, mod-$3$ image \href{https://www.lmfdb.org/knowledge/show/gsp4.subgroup_data?label=3.25920.3}{3.25920.3} and $3$-torsion field $\Q(\zeta_{12})$.
    The curve $C$ is not currently in LMFDB.
\end{example}

\bibliographystyle{amsplain}
\bibliography{Mod3GaloisImage.bib}

\providecommand{\bysame}{\leavevmode\hbox to3em{\hrulefill}\thinspace}
\providecommand{\MR}{\relax\ifhmode\unskip\space\fi MR }
\providecommand{\MRhref}[2]{%
  \href{http://www.ams.org/mathscinet-getitem?mr=#1}{#2}
}
\providecommand{\href}[2]{#2}
\begin{thebibliography}{10}

\bibitem{gpsgsp4labelknowl}
\emph{Labelling subgroups of {GS}p(4,$\ell$)}, \url{https://www.lmfdb.org/knowledge/show/g2c.galois_rep_modell_image}, 2025, [Online; accessed 18 June 2025].

\bibitem{BBKKMSV24}
Barinder~S. Banwait, Armand Brumer, Hyun~Jong Kim, Zev Klagsbrun, Jacob Mayle, Padmavathi Srinivasan, and Isabel Vogt, \emph{Computing nonsurjective primes associated to {G}alois representations of genus {$2$} curves}, Lu{C}a{NT}: {LMFDB}, computation, and number theory, Contemp. Math., vol. 796, Amer. Math. Soc., [Providence], RI, [2024] \copyright 2024, pp.~129--163. \MR{4732686}

\bibitem{Magma}
Wieb Bosma, John Cannon, and Catherine Playoust, \emph{The {M}agma algebra system. {I}. {T}he user language}, J. Symbolic Comput. \textbf{24} (1997), no.~3-4, 235--265, Computational algebra and number theory (London, 1993). \MR{MR1484478}

\bibitem{BCGP}
George Boxer, Frank Calegari, Toby Gee, and Vincent Pilloni, \emph{Abelian surfaces over totally real fields are potentially modular}, Publ. Math. Inst. Hautes \'{E}tudes Sci. \textbf{134} (2021), 153--501. \MR{4349242}

\bibitem{BF22}
Nils Bruin and Eugene Filatov, \emph{Twists of the {B}urkhardt quartic threefold}, Res. Number Theory \textbf{8} (2022), no.~4, Paper No. 73, 16. \MR{4487859}

\bibitem{CC}
Frank Calegari and Shiva Chidambaram, \emph{Rationality of twists of the {S}iegel modular variety of genus 2 and level 3}, Proc. Amer. Math. Soc. \textbf{150} (2022), no.~5, 1975--1984. \MR{4392333}

\bibitem{CCG}
Frank Calegari, Shiva Chidambaram, and Alexandru Ghitza, \emph{Some modular abelian surfaces}, Math. Comp. \textbf{89} (2020), no.~321, 387--394. \MR{4011548}

\bibitem{cassels_flynn_1996}
J.~W.~S. Cassels and E.~V. Flynn, \emph{Prolegomena to a middlebrow arithmetic of curves of genus 2}, London Mathematical Society Lecture Note Series, Cambridge University Press, 1996.

\bibitem{Chidambaram_Mod3Image}
Shiva Chidambaram, \emph{{Mod3Image}}.

\bibitem{Die02}
Luis~V. Dieulefait, \emph{Explicit determination of the images of the {G}alois representations attached to abelian surfaces with {${\rm End}(A)=\Bbb Z$}}, Experiment. Math. \textbf{11} (2002), no.~4, 503--512. \MR{1969642}

\bibitem{LO77}
J.~C. Lagarias and A.~M. Odlyzko, \emph{Effective versions of the {C}hebotarev density theorem}, Algebraic number fields: {$L$}-functions and {G}alois properties ({P}roc. {S}ympos., {U}niv. {D}urham, {D}urham, 1975), Academic Press, London-New York, 1977, pp.~409--464. \MR{447191}

\bibitem{lmfdb}
The {LMFDB Collaboration}, \emph{The {L}-functions and modular forms database}, \url{https://www.lmfdb.org}, 2024, [Online; accessed 18 June 2025].

\bibitem{RSZB22}
Jeremy Rouse, Andrew~V. Sutherland, and David Zureick-Brown, \emph{{$\ell$}-adic images of {G}alois for elliptic curves over {$\Bbb{Q}$} (and an appendix with {J}ohn {V}oight)}, Forum Math. Sigma \textbf{10} (2022), Paper No. e62, 63, With an appendix with John Voight. \MR{4468989}

\bibitem{Sut16}
Andrew~V. Sutherland, \emph{Computing images of {G}alois representations attached to elliptic curves}, Forum Math. Sigma \textbf{4} (2016), Paper No. e4, 79. \MR{3482279}

\bibitem{Sut21}
\bysame, \emph{Stronger arithmetic equivalence}, Discrete Anal. (2021), Paper No. 23, 23. \MR{4341956}

\bibitem{vBCCK24}
Raymond van Bommel, Shiva Chidambaram, Edgar Costa, and Jean Kieffer, \emph{Computing isogeny classes of typical principally polarized abelian surfaces over the rationals}, Lu{C}a{NT}: {LMFDB}, computation, and number theory, Contemp. Math., vol. 796, Amer. Math. Soc., [Providence], RI, [2024] \copyright 2024, pp.~187--214. \MR{4732688}

\bibitem{Zyw22}
David Zywina, \emph{Explicit open images for elliptic curves over $\mathbb{Q}$}, 2022.

\end{thebibliography}

\end{document}